\documentclass{article}
\usepackage[T2A]{fontenc}
\usepackage[utf8]{inputenc}
\usepackage[english,russian]{babel}
\usepackage[tbtags]{amsmath}
\usepackage{amsfonts,amssymb,mathrsfs,amscd}
\usepackage[hyper]{msb-a}

\makeatletter
\gdef\No{{\select@language{russian}\textnumero}}
\makeatother

\JournalName{}
\numberwithin{equation}{section}
\theoremstyle{plain}
\newtheorem{theorem}{Теорема}
\newtheorem{lemma}{Лемма}[section]

\theoremstyle{definition}

\newtheorem{proof}{Доказательство}

\begin{document}

\newcommand{\mdp}{\mathrm{mdp}}

\title{Нижние оценки чисел независимости дистанционных графов с вершинами в $\{-1, 0, 1\}^n$}

\author[A.\,R. ~Akhiiarov]{А.\,Р. ~Ахияров}
\address{Московский физико-технический институт (государственный университет)}
\email{akhiiarov.ar@phystech.edu}

\author[A.\,V. ~Bobu]{А.\,В. ~Бобу}
\address{Criteo S.\,A., Франция}
\email{a.v.bobu@gmail.com}

\author[A.\,M. ~Raigorodskii]{А.\,М. ~Райгородский}
\address{Московский физико-технический институт (государственный университет), кафедра дискретной математики и лаборатория продвинутой комбинаторики и сетевых приложений;\\ 
МГУ им. М.В. Ломоносова, механико-математический факультет, кафедра математической статистики и случайных процессов; \\
Кавказский математический центр Адыгейского государственного университета;\\ 
Бурятский государственный университет, институт математики и информатики}
\email{mraigor@yandex.ru}

\date{}
\udk{519.157.4+519.176}

\maketitle
\begin{fulltext}

\begin{abstract}
    Данная работа посвящена нижним оценкам чисел независимости дистанционных графов с вершинами в $\{-1,0,1\}^n$. Изучен асимптотический случай, получены новые результаты в широком диапазоне параметров. Приведены численные результаты, демонстрирующие нетривиальные соотношения между полученными оценками. Отдельно обсуждаются известные верхние оценки и их возможная неоптимальность. 

    Библиография: 33 названия. 
\end{abstract}


\markright{Числа независимости графов с вершинами в $\{-1, 0, 1\}^n$}

\section{Введение}

Одной из наиболее интересных и сложных задач комбинаторной геометрии является проблема Нелсона--Эрдеша--Хадвигера, которая получила широкую известность в середине XX века (см. \cite{Had}--\cite{Rai3}). Она заключается в отыскании величины $\chi(\mathbb R^n)$ --- хроматического числа пространства, то есть минимального количества цветов, в которые можно так покрасить все точки $\mathbb R^n$, чтобы любые две точки на расстоянии 1 были разных цветов. Удивительно, что задача по-прежнему далека от разрешения даже для случая плоскости: больше полувека было известно лишь, что $4 \leqslant \chi(\mathbb R^2) \leqslant 7$, причем обе оценки почти тривиальны. Тем неожиданнее недавний прорыв: в работе \cite{deGrey} Обри де Грей показал, что $\chi(\mathbb R^2) \geqslant 5$. 

В более высоких размерностях задача становится лишь сложнее: зазоры между нижней и верхней оценками увеличиваются с ростом $n$. В работе \cite{CKR} приведен достаточно актуальный список нижних оценок $\chi(\mathbb R^n)$ при $n \leqslant 12$, а в обзоре \cite{Rai2} указаны верхние оценки при $n \leqslant 4$.\footnote{Насколько нам известно, приведенные в этих работах оценки по-прежнему являются наилучшими известными.} В асимптотическом случае, при $n\to\infty$, легко получить оценку $\chi(\mathbb R^n) \leqslant \left( \sqrt n + 2 \right)^n$, а первая, по-видимому, нетривиальная граница была получена в работе \cite{LR} 1972 года (см. также \cite{P}):
$$
    \chi(\mathbb R^n) \leqslant (3 + o(1))^n, \quad n\to\infty.
$$
Эта оценка остается наилучшей известной на настоящий момент. Что касается нижних границ, то первые нелинейные оценки появились в работе \cite{LR} в 1972 году. В ней по сути было использовано семейство дистанционных графов\footnote{Дистанционным назовем граф, вершины которого --- точки пространства, а ребра --- пары точек на заданном расстоянии $a>0$.} $G_n(k,t) = (V_n(k), E_n(k,t))$, у которых 
\begin{gather*}
    V_n(k) = \{ \mathbf x = (x_1, \ldots, x_n) \colon x_i \in \{0, 1\}, x_1 + \ldots + x_n = k \}, \\ 
    E_n(k,t) = \left\{ \{\mathbf x, \mathbf y\} \colon |\mathbf x - \mathbf y| = \sqrt{2(k-t)}  \right\} = \left\{ \{\mathbf x, \mathbf y\} \colon (\mathbf x, \mathbf y) = t \right\}. 
\end{gather*}
Напомним, что \textit{хроматическое число} графа $G$ --- это наименьшее число $\chi(G)$ цветов, в которые можно так покрасить все вершины графа, чтобы между вершинами одного цвета не было ребра. Ясно, что $\chi(\mathbb R^n) \geqslant \chi(G_n(k,t))$. В свою очередь $\displaystyle \chi(G_n(k,t)) \geqslant \frac{|V(G_n(k,t))|}{\alpha(G_n(k,t))}$, где $\alpha(G)$ --- \textit{число независимости} графа $G$ --- максимальный размер множества вершин, которые попарно не соединены ребром. Поэтому нижние границы из статьи \cite{LR} опирались на верхние оценки чисел независимости $\alpha(G_n(k, t))$.

Величина $m(n, k, t) = \alpha(G_n(k, t))$ впоследствии изучалась в большом количестве работ (см. \cite{CKR}, \cite{Nagy}--\cite{BobuKuprRai}). Отдельно хотелось бы отметить работу \cite{FW}, в которой авторы при помощи новаторского линейно-алгебраического метода получили, среди прочего, верхние оценки величины $\displaystyle m\left(n, k, t\right)$ при некоторых специально выбранных значениях $k$ и $t$ и, как следствие, экспоненциальную нижнюю границу для $\chi(\mathbb R^n)$: 
$$
    \chi(\mathbb R^n) \geqslant \left( \frac{1+\sqrt 2}2 + o(1) \right)^n \geqslant (1.207\ldots + o(1))^n.
$$ 
Поскольку проблема отыскания величины $m(n,k,t)$ мотивирована приведенными выше результатами, при ее изучении часто считается, что $n$ --- основной параметр, а $k$ и $t$ --- некоторые функции от $n$. При этом обычно предполагается, что $n\to\infty$. Параметры $k$ и $t$ в этом случае могут вести себя по-разному относительно $n$. Если $k$ и $t$ --- константы, не зависящие от $n$, то известно, что верхняя оценка работы \cite{FW} асимптотически точна. Если же параметры $k$ и $t$ растут линейно по сравнению с $n$, ситуация намного сложнее: при $2t \leqslant k$ найдена лишь асимптотика $\ln m(n,k,t)$, а при $2t > k$ даже эта асимптотика по-прежнему неизвестна (см. \cite{PonRai1}, \cite{BobuKuprRai}).

Развитие линейно-алгебраического метода привело к тому, что в работе \cite{Rai1} была улучшена нижняя оценка $\chi(\mathbb R^n)$:
\begin{equation}\label{eq:chi-R-n-lower}
    \chi(\mathbb R^n) \geqslant (1.239\ldots + o(1))^n.
\end{equation}
Идея улучшения заключается в использовании графов 
$$G_n(k_{-1}, k_0, k_1, t)=(V_n(k_{-1}, k_0, k_1), E_n(k_{-1}, k_0, k_1, t)),$$
похожих на $G_n(k,t)$, но с вершинами в $\{-1, 0, 1\}^n$: 
\begin{gather*}
  V_n(k_{-1}, k_0, k_1) = \{\mathbf{x} = (x_1, \ldots, x_n) \colon x_i \in \{-1, 0, 1\}, |\{i \colon x_i = -1\}| = k_{-1}, \\ |\{i \colon x_i = 0\}| = k_0, |\{i \colon x_i = 1\}|=k_1\}, \:
  E_n(k_{-1}, k_0, k_1,t) = \left\{ \{\mathbf x, \mathbf y\} \colon (\mathbf x, \mathbf y) = t \right\}.
\end{gather*}

При получении границы \eqref{eq:chi-R-n-lower} использовалась верхняя оценка величины \\ $\alpha(G_n(k_{-1}, k_0, k_1, t))$ при некоторых специально выбранных значениях параметров. Тем самым данная работа в том числе положила начало изучению величины 
\begin{gather*}
    m(n, k_{-1}, k_0, k_1, t) = \alpha(G_n(k_{-1}, k_0, k_1, t)) = \\ 
    = \max\{|W|: W \subset V_n(k_{-1}, k_0, k_1), (\mathbf{x}, \mathbf{y}) \neq t \text{ для всех } \mathbf{x}, \mathbf{y} \in W \}.
\end{gather*}
Впоследствии отысканию величины $m(n, k_{-1}, k_0, k_1, t)$ было посвящено значительное число работ. Например, статьи \cite{KhR}--\cite{MosRai} дают содержательные нижние оценки $m(n, k_{-1}, k_0, k_1, t)$, часть из которых мы приведем в следующем разделе. Работа \cite{PonRai2}, напротив, дает верхние оценки $m(n, k_{-1}, k_0, k_1, t)$, которые применимы в широком диапазоне значений параметров. Также ряд нетривиальных оценок получен в работах \cite{FrKup1}--\cite{FrKup5}.

Целью настоящей работы будет исследование величины $m(n, k_{-1}, k_0, k_1, t)$. В первую очередь мы сосредоточимся на получении нижних оценок для широкого диапазона значений $k_{-1}, k_0, k_1$ и $t$. Мы ограничимся случаем, когда $n\to\infty$, а $k_{-1}, k_0, k_1$ и $t$ растут линейно по сравнению с $n$. Более формально, пусть числа $k_{-1}', k_0', k_1', t' \in (0, 1)$.  Будем считать, что $k_{-1} \sim k_{-1}'n, k_0 \sim k_0'n, k_1 \sim k_1'n, t \sim t'n$ при $n \to \infty$ и $k_{-1} + k_0 + k_1 = n$. Оказывается, что в этом случае и нижние, и верхние оценки экспоненциальные, то есть вида $(\lambda(k'_{-1}, k'_{0}, k'_1, t') + o(1))^n$, где значение $\lambda$ не зависит от $n$. Мы будем считать, что две оценки различаются сильно, если имеют разные величины $\lambda$. 

В следующем разделе мы укажем известные на настоящий момент оценки величины $m(n, k_{-1}, k_0, k_1, t)$. В разделе \ref{sect:new-results} мы приведем формулировки полученных нами результатов. Раздел \ref{sect:numerical-results} будет посвящен численному сравнению новых и уже известных оценок. Наконец, раздел \ref{sect:proofs} содержит доказательства полученных нами результатов.

\section{Известные оценки}\label{sect:KnownBounds}

Прежде чем переходить к новым результатам, содержащим нижние оценки величины $m(n, k_{-1}, k_0, k_1, t)$, приведем известные на настоящий момент ее верхние оценки. Первую из них можно отыскать в книге \cite{RaiLinAlgebra}.  

\begin{theorem}\label{th:FW}
    Пусть $n, k_{-1}, k_0, k_1, t$ таковы, что $k_1 + k_{-1} \leqslant n/2$, $k_{-1} \leqslant k_1$, разность $q = k_1 + k_{-1} - t$ является степенью простого числа и $k_1 + k_{-1} - 2q < -2k_{-1}$. Тогда 
    $$  
        m(n, k_{-1}, k_0, k_1, t) \leqslant \sum\limits_{(i,j) \in \mathcal A} C_n^i C_{n-i}^j, 
    $$
    где 
    $$
        \mathcal A = \{(i,j) \colon i+j\leqslant n, i+2j \leqslant q-1\}.
    $$
\end{theorem}  
Кроме того, приведем оценку, первый случай которой по сути был доказан в работе \cite{Rai} в виде леммы. Второй случай теоремы был получен на основе той же идеи и новейшего результата из \cite{PonRai1}, поэтому в разделе \ref{sect:FW-plus-flower-proof} мы приведем его доказательство.

\begin{theorem}\label{th:FW-plus-flower}
    Пусть $k_1 \leqslant \frac{n - k_{-1}}{2}$, $k_{-1} \leqslant t$ и $q = k_1 + k_{-1} - t$ является степенью простого числа. Тогда имеют место два случая:
    \begin{enumerate}
        \item Если $2(t - k_{-1}) < k_1$, то
        $$
            m(n, k_{-1}, k_0, k_1, t) \leqslant C_n^{k_{-1}} \sum\limits_{i = 0}^{q-1} C_{k_1 + k_0}^i. 
        $$
        \item Если $2(t - k_{-1}) \geqslant k_1$, то положим $d = 2(t - k_{-1}) - k_1 + 1$. Далее, пусть натуральные $d_1, d_2$ таковы, что $d_1 + d_2 = d$. Положим $n_1 = (n - k_{-1}) - d_1, k_2 = k_1 - d_1$. Определим натуральное число $r$ из соотношения 
        $$
            (k_2 - d_2 + 1)\left(2 + \frac{d_2 - 1}{r + 1}\right) \leqslant n_1 < (k_2 - d_2 + 1)\left(2 + \frac{d_2 - 1}{r}\right).
        $$
        Тогда 
        $$
            m(n, k_{-1}, k_0, k_1, t) \leqslant C_n^{k_{-1}} \frac{ C^{d_2 + 2r}_{n_1} C^{d_1}_{n - k_{-1}}}{ C^{d_2 + r}_{k_2} C^{r}_{n_1 - k_2} C^{d_1}_{k_1} } \left(\sum\limits_{i = 0}^{q-1} C_{n_1}^i\right).
        $$
    \end{enumerate}
\end{theorem}

Далее нам будет полезно следующее обозначение. Введем функцию $\mdp$ (minimum dot product), вычисляющую минимально возможное скалярное произведение двух векторов из $V_n(k_{-1}, k_0, k_1)$. В параграфе \ref{sect:lem-mdp} (лемма \ref{lem:mdp}) будет доказано, что эта функция имеет довольно простой вид: 
\begin{equation*}
  \mdp(k_{-1}, k_0, k_1) =
    \begin{cases}
      -2\min(k_{-1}, k_1), \: \text{если $k_0 \geqslant \max(k_{-1}, k_1) - \min(k_{-1}, k_1)$,} \\
      \max(k_{-1}, k_1)  - 3\min(k_{-1}, k_1) - k_0, \: \text{иначе.}
    \end{cases}       
\end{equation*}

Приведем результат из работы \cite{PonRai2}, в некотором смысле дополняющий теорему \ref{th:FW}.

\begin{theorem}\label{th:PonRai} 
    Пусть $n, k_{-1}, k_0, k_1, t$ таковы, что $k_1 + k_{-1} \leqslant n/2$, $k_{-1} \leqslant k_1$, разность $q = k_1 + k_{-1} - t$ является степенью простого числа и $k_1 + k_{-1} - 2q \geqslant -2k_{-1}$. Положим $ d = k_1+k_{-1}-2q+1 $. Пусть неотрицательные целые числа 
$$
    m_{-1}, \: m_0, \: m_1, \: m_{-1, -1}, \: m_{0,-1}, \: m_{1, -1}, \: m_{-1, 0}, \: m_{0,0}, 
 \: m_{1,0}, \: m_{-1,1}, \: m_{0,1}, \: m_{1,1}
 $$
 удовлетворяют ограничениям
 \begin{gather*}
   m_{-1} + m_0 + m_1 = n, \\
   m_{\alpha, -1} + m_{\alpha, 0} + m_{\alpha, 1} = k_{\alpha} \text{ для всех } \alpha \in \{-1, 0, 1\}, \\
   m_{-1, \beta} + m_{0,\beta} + m_{1, \beta} = m_{\beta} \text{ для всех } \beta \in \{-1, 0, 1\}, \\
   \sum^1_{\beta = -1} \mdp(m_{-1, \beta}, m_{0, \beta}, m_{1, \beta}) \geqslant d.
 \end{gather*}
Тогда
\begin{gather*}\label{eq:upper-bound-final-expression}
    m(n,k_{-1}, k_0, k_1, t) \leqslant \frac{n!}{m_{-1}!m_0!m_1!} \cdot \frac{m_{-1,-1}! m_{-1,0}! m_{-1,1}!}{k_{-1}!} \cdot \frac{m_{0,-1}! m_{0,0}! m_{0,1}!}{k_0!} \cdot \\ \frac{m_{1,-1}! m_{1,0}! m_{1,1}!}{k_1!} \cdot \left(\sum_{(i,j) \in {\cal A}} C_n^i C_{n-i}^j\right),
\end{gather*}
где
$$
{\cal A} = \{(i,j) \colon i + j \leqslant n, \, i+2j \leqslant q-1\}.
$$
\end{theorem}
Отметим, что первые четыре множителя в правой части \eqref{eq:upper-bound-final-expression} можно было бы переписать в следующем виде: 
\begin{gather*}
    \frac{n!}{m_{-1}!m_0!m_1!} \cdot \frac{m_{-1,-1}! m_{-1,0}! m_{-1,1}!}{k_{-1}!} \cdot \frac{m_{0,-1}! m_{0,0}! m_{0,1}!}{k_0!} \cdot \frac{m_{1,-1}! m_{1,0}! m_{1,1}!}{k_1!} = \\ 
    = \frac{C_n^{m_{-1}} C_{m_0 + m_1}^{m_0}}{\prod\limits_{\alpha = -1}^1 C_{k_{\alpha}}^{m_{\alpha, -1}} C_{m_{\alpha, 0} + m_{\alpha, 1}}^{m_{\alpha, 1}}}. 
\end{gather*}
Мы будем активно использовать подобную форму записи в формулировках следующих теорем. 

Сформулируем лучшую известную на настоящий момент нижнюю оценку, но для большей ясности сперва опишем похожую идею для случая $(0,1)$-векторов. Как следует из работы \cite{BobuKuprRai}, для величины $m(n,k,t)$ одной из самых простых конструкций в случае нижних оценок оказалась конструкция Альсведе--Хачатряна (см. \cite{AKh1}--\cite{AB}). Основная идея ее построения такова. Исходное множество элементов $\{1, \ldots, n\}$ разбивается на две, вообще говоря, неравные части. В качестве итоговой совокупности мы выбираем векторы с $k$ единицами так, чтобы уже в одной из этих частей их попарные скалярные произведения оказались больше $t$. Это можно сделать, если, например, набирать в совокупность только те векторы, у которых количество единиц в первой из частей больше некоторой заранее выбранной величины. Небольшая сложность заключается лишь в том, чтобы правильно подобрать эту величину и аккуратно доказать, что количество векторов в такой конструкции действительно растет экспоненциально с ростом $n$.

Такая конструкция без труда обобщается на случай $(-1, 0, 1)$-векторов, однако формулировка теоремы становится более громоздкой. Сперва сформулируем результат из работы \cite{GLRU2}, а затем прокомментируем использованные параметры. 
\begin{theorem}\label{th:decreasing}
    Пусть фиксированы неотрицательные целые числа 
$$m_{-1}, \: m_0, \: m_{1}, \: m_{-1, -1}, \: m_{0,-1}, \: m_{1, -1}, \: m_{-1, 0}, \: m_{0,0}, 
 \: m_{1,0}, \: m_{-1,1}, \: m_{0,1}, \: m_{1,1},$$
 удовлетворяющие ограничениям теоремы \ref{th:PonRai} и условию на минимальное скалярное произведение с параметром $t$: 
 \begin{gather*}
   m_{-1} + m_0 + m_1 = n,\\  
   m_{\alpha, -1} + m_{\alpha, 0} + m_{\alpha, 1} = k_{\alpha} \text{ для всех } \alpha \in \{-1, 0, 1\}, \\
   m_{-1, \beta} + m_{0,\beta} + m_{1, \beta} = m_{\beta} \text{ для всех } \beta \in \{-1, 0, 1\},\\
   \sum^1_{\beta = -1} \mdp(m_{-1, \beta}, m_{0, \beta}, m_{1, \beta}) > t.
 \end{gather*}
Тогда
$$m(n, k_{-1}, k_0, k_1, t) \geqslant \prod^1_{\beta = -1} C^{m_{-1, \beta}}_{m_{\beta}} C^{m_{0, \beta}}_{m_{0, \beta} + m_{1, \beta}}.$$
\end{theorem}

В отличие от случая $(0,1)$-векторов теперь более выгодным оказывается разбиение исходного множества на три части: $\{1, \ldots, n\} = M_{-1} \sqcup M_0 \sqcup M_1$. Если обозначить $|M_{\beta}| = m_{\beta}, \beta\in \{-1,0,1\}$, то ясно, что $m_{-1}+m_0+m_1=n$, как и указано в теореме выше. Далее, будем набирать в совокупность векторы, у которых в каждой части $M_{\beta}$ содержится $m_{1, \beta}$ единиц, $m_{0, \beta}$ нулей и $m_{-1, \beta}$ минус единиц. Несложно понять, что тогда в качестве чисел $m_{\alpha, \beta}$ (где $\alpha, \beta \in \{-1, 0, 1\}$) можно взять лишь те, которые связаны уравнениями
\begin{gather*}
   m_{\alpha, -1} + m_{\alpha, 0} + m_{\alpha, 1} = k_{\alpha} \text{ для всех } \alpha \in \{-1, 0, 1\}, \\
   m_{-1, \beta} + m_{0,\beta} + m_{1, \beta} = m_{\beta} \text{ для всех } \beta \in \{-1, 0, 1\}.
\end{gather*}
Наконец, последнее условие является наиболее содержательным. Понятно, что в каждой из частей $M_{\beta}$ можно посчитать минимально возможное скалярное произведение пары ``подвекторов'' (состоящих из $m_{\beta}$ координат): оно равно $\mdp(m_{-1, \beta}, m_{0, \beta}, m_{1, \beta})$. Если окажется, что сумма таких минимальных скалярных произведений больше $t$, то запрет никогда не реализуется, и размер полученной совокупности векторов служит нижней оценкой величины $m(n, k_{-1}, k_0, k_1, t)$.  Подсчитать размер построенной конструкции довольно просто: достаточно лишь оценить число способов выбрать векторы по описанному выше правилу. Мы будем называть полученную совокупность \textit{конструкцией Альсведе--Хачатряна}. 

Остальные нижние оценки из работ \cite{GLRU1}--\cite{MosRai} являются по сути частными случаями теоремы \ref{th:decreasing}, мы не будем их приводить отдельно. Что касается работы \cite{BobuKuprRai} (посвященной случаю $(0, 1)$-векторов), то, как мы увидим в следующем разделе, почти все оценки из нее удается обобщить на случай $(-1, 0, 1)$-векторов. Коль скоро наши новые оценки в следующем разделе более общие, формулировать специальные теоремы для случая $k_{-1} = 0$ мы сочли избыточным. 

\section{Формулировки новых результатов}\label{sect:new-results}

Первая из полученных нами оценок является обобщением теоремы 4 из работы \cite{BobuKuprRai}. Для формулировки теоремы нам понадобится определение величины $h(n, k_{-1}, k_0, k_1, t)$, которая задает размер максимальной совокупности векторов с попарными скалярными произведениями меньше $t$:
$$
    h(n, k_{-1}, k_0, k_1, t) = \text{max}\{|W|: W \subset V_n(k_{-1}, k_0, k_1), \: (\mathbf{x}, \mathbf{y}) < t \text{ для всех } \mathbf{x}, \mathbf{y} \in W \}.
$$

Теперь сформулируем полученный результат, а затем изложим основную идею доказательства, которое будет полноценно изложено в параграфе \ref{sect:VG-proof}. 
\begin{theorem}\label{th:VG}
    Пусть $k_{-1}, k_0, k_1$ таковы, что следующее выражение не равно нулю: 
    \begin{gather*}\label{eq:d-def}
        d(n, k_{-1}, k_0, k_1, t) = \sum\limits_{(l_{1,1}, l_{-1,-1}, l_{-1,1}, l_{1,-1}) \in \mathcal B} C^{l_{1,1}}_{k_1} C^{l_{1,-1}}_{k_1-l_{1,1}} C^{l_{-1,-1}}_{k_{-1}} C^{l_{-1,1}}_{k_{-1}-l_{-1,-1}} \\ C^{k_{1} - l_{1,1} - l_{-1,1}}_{k_{0}} C^{k_{-1} - l_{-1,-1} - l_{1,-1}}_{k_{0} - k_{1} + l_{1,1} + l_{-1,1}}  P(l_{1,1}, l_{-1,1}, l_{-1,-1}, l_{1,-1}, t).
    \end{gather*}
    Здесь 
    \begin{equation*}
  P(l_{1,1}, l_{-1,1}, l_{-1,-1}, l_{1,-1}, t) =
    \begin{cases}
      1, & \text{если $l_{1,1} - l_{-1,1}  + l_{-1,-1} - l_{1,-1} \geqslant t$,} \\
      0, & \text{иначе;}
    \end{cases}       
    \end{equation*}
    и 
    \begin{multline*}
        \mathcal B = \bigl\{(l_{1,1}, l_{-1,-1}, l_{-1,1}, l_{1,-1}) \colon 0 \leqslant l_{1,1} \leqslant k_1; 0 \leqslant l_{-1,-1} \leqslant k_{-1}; \\ 
        0 \leqslant l_{-1,1} \leqslant \min(k_{-1} - l_{-1, -1}, k_1 - l_{1,1}); 0 \leqslant l_{1,-1} \leqslant \min(k_{-1} - l_{-1, -1}, k_1 - l_{1,1}) \bigr\}.
    \end{multline*}    
    Тогда 
    $$
        m(n, k_{-1}, k_0, k_1, t) \geqslant h(n, k_{-1}, k_0, k_1, t) \geqslant \left\lceil\frac{|V(n,k_{-1}, k_0, k_1)|}{d(n, k_{-1}, k_0, k_1, t)} \right\rceil. 
    $$
\end{theorem}

Конструкция, используемая в доказательстве данной теоремы, состоит из векторов, попарные скалярные произведения которых меньше $t$. Идея построения проста и может быть описана, например, таким образом. Конструкцию будем строить итеративно, добавляя на каждом шаге один вектор. Зафиксируем произвольный вектор $v \in V_n(k_{-1}k_0, k_1)$ и добавим его в итоговую совокупность. Уберем из рассмотрения все ``плохие'' векторы, скалярное произведение которых с первым хотя бы $t$. Число таких векторов, как будет показано в доказательстве, равно знаменателю полученной оценки. Из оставшихся векторов выберем любой и добавим его в совокупность. Очевидно, для второго вектора количество ``плохих'' векторов будет также не больше знаменателя оценки. Таким образом, можно жадно добавлять каждый следующий вектор в совокупность, пока общее число ``плохих'' векторов не достигнет числа всех вершин $|V_n(k_{-1}k_0, k_1)|$. 
Данная конструкция используется в доказательстве классической для теории кодирования границы Варшамова--Гилберта (см. \cite{V}--\cite{MS}). Поэтому (а также следуя обозначениям работы \cite{BobuKuprRai}) мы будем называть полученную совокупность \textit{конструкцией Варшамова--Гилберта}. 

Следующая теорема соединяет идеи конструкций Варшамова--Гилберта и Альсведе--Хачатряна. 
\begin{theorem}\label{th:Glue}
    Пусть $t_1 \leqslant t$ и фиксированы неотрицательные целые числа 
$$m_{-1}, \: m_0, \: m_1, \: m_{-1, -1}, \: m_{0,-1}, \: m_{1, -1}, \: m_{-1, 0}, \: m_{0,0}, 
 \: m_{1,0}, \: m_{-1,1}, \: m_{0,1}, \: m_{1,1},$$
удовлетворяющие ограничениям
\begin{gather*}
    m_{-1} + m_0 + m_1 = n, \\
    m_{\alpha, -1} + m_{\alpha, 0} + m_{\alpha, 1} = k_{\alpha} \text{ для всех } \alpha \in \{-1, 0, 1\}, \\
   m_{-1, \beta} + m_{0,\beta} + m_{1, \beta} = m_{\beta} \text{ для всех } \beta \in \{-1, 0, 1\}, \\
      \sum^1_{\beta = -1} \mdp(m_{-1, \beta}, m_{0, \beta}, m_{1, \beta}) > t, \\
    t_1 + 2 \cdot extras(m_{0, -1}, m_{0, 1}, m_{1, -1}, m_{1, 1},  m_{-1, -1}, m_{-1, 1}, m_{-1}, m_1) + 2(m_{1, 0} + m_{-1, 0}) \leqslant t,
\end{gather*}
где 

\begin{multline*}
 extras(m_{0, -1}, m_{0, 1}, m_{1, -1}, m_{1, 1},  m_{-1, -1}, m_{-1, 1}, m_{-1}, m_1) = \min(E_1, E_2, E_3) \text{ и} \\
 E_1 = 2\min(m_{-1}, m_1), \\
 E_2 = \min(m_{-1, 1}, m_{-1, -1}) + \min(m_{1, 1}, m_{1, -1}) + \min(m_{-1}, m_1), \\
 E_3 = \min(m_{-1, 1}, m_{-1, -1}) + \min(m_{1, 1}, m_{1, -1}) + \min(m_{1, 1} + m_{-1, -1}, m_{-1, 1} + m_{1, -1}) + \\ m_{0, 1} + m_{0, -1}.
\end{multline*}

Тогда выполнена оценка
$$
m(n, k_{-1}, k_0, k_1, t) \geqslant h(n, m_{-1}, m_0, m_1, t_1) \prod^1_{\beta = -1} C^{m_{-1, \beta}}_{m_{\beta}} C^{m_{0, \beta}}_{m_{0, \beta} + m_{1, \beta}}.
$$
\end{theorem}

Эта теорема будет доказана в параграфе \ref{sect:Glue-proof}. Идея доказательства состоит в следующем. Сначала построим конструкцию Вар\-ша\-мо\-ва--Гил\-бер\-та $\mathcal F$, в которой попарные скалярные произведения векторов не будут превышать $t_1$. Такая конструкция существует по теореме \ref{th:VG}. Будем считать, что каждый $\mathbf x \in \mathcal F$ задает значением своих координат некоторое разбиение $\{1, \ldots, n\} = M^{\mathbf x}_{-1} \sqcup M^{\mathbf x}_0 \sqcup M^{\mathbf x}_1$. Итоговая совокупность будет объединением совокупностей векторов, удовлетворяющих разбиению из теоремы \ref{th:decreasing} хотя бы для одного $\mathbf x \in \mathcal F$. Таким образом мы создаем много конструкций Альсведе--Хачатряна, в каждой из которых, как известно, все попарные скалярные произведения больше $t$. Если подобрать параметр $t_1$ неким специальным способом, то можно будет гарантировать, что векторы из разных конструкций Альсведе--Хачатряна имеют попарные скалярные произведения меньше $t$, и запрет не реализуется.

Наконец, последняя из полученных нами теорем развивает идеи теоремы \ref{th:Glue}.

\begin{theorem}\label{th:SuperGlue}
    Пусть неотрицательные целые числа 
$$t_1, m_{-1}, \: m_0, \: m_1, \: m_{-1, -1}, \: m_{0,-1}, \: m_{1, -1}, \: m_{-1, 0}, \: m_{0,0}, 
 \: m_{1,0}, \: m_{-1,1}, \: m_{0,1}, \: m_{1,1}$$
удовлетворяют ограничениям теоремы \ref{th:Glue}. 

Пусть $t_1 < s < m_{-1} + m_1$ и $t - 2(m_{1, 0} + m_{-1, 0}) \geqslant 0$. Тогда выполнена оценка 
\begin{multline*}
    m(n, k_{-1}, k_0, k_1, t) \geqslant h(n, m_{-1}, m_0, m_1, s) \times \\ 
    \times \left( \prod^1_{\beta = -1} C^{m_{-1, \beta}}_{m_{\beta}} C^{m_{0, \beta}}_{m_{0, \beta} + m_{1, \beta}} - h(n, m_{-1}, m_0, m_1, s) C_{m_0}^{m_{-1, 0}} C_{m_{0,0} + m_{1,0}}^{m_{0,0}} R \right),
\end{multline*}
где 
\begin{multline*}
    R = \max_{m_{-1} + m_1 - m_0 \leqslant l \leqslant \min(s + 4\min(m_{-1}, m_1), m_{-1} + m_1)} \sum\limits_{j = t - 2(m_{1,0} + m_{-1, 0})}^l \sum\limits_{i=0}^j C_l^j C_j^i \times \\
    \times C_{m_{-1} + m_1 - l}^{m_{1, 1} + m_{1, -1} + m_{-1, -1} + m_{-1, 1} - j} C_{m_{1, 1} + m_{1, -1} + m_{-1, -1} + m_{-1, 1} - j}^{m_{1, 1} + m_{1, -1} - i}.
\end{multline*}
\end{theorem}

Данная теорема будет доказана в параграфе \ref{sect:proof-superglue}. Идея доказательства в значительной мере напоминает идею доказательства теоремы \ref{th:Glue}. Однако теперь мы будем использовать совокупность $\mathcal F$, в которой попарные скалярные произведения меньше $s$, причем величина $s$ может оказаться больше $t_1$. Таким образом, в итоговой совокупности окажется некоторое количество ``плохих'' векторов, скалярные произведения которых равны $t$. Поэтому требуется дополнительное действие, чтобы ``проредить'' данную совокупность и удалить все ``плохие'' векторы. Несмотря на простоту самой идеи, оценка размера полученной совокупности оказывается технически более сложной. 

Удивительно, что есть еще один результат, который формулируется и доказывается очень просто, но при этом дает при некоторых значениях параметров лучшие оценки. 
\begin{theorem}\label{th:Pairs}
    Пусть $k_1 + k_{-1}$ и $n$ делятся на $2$, а $t$ \,--\, нечетное число. Тогда 
    $$
        m(n, k_{-1}, k_0, k_1, t) \geqslant C^{(k_1 + k_{-1})/2}_{n/2} C^{k_1}_{k_1 + k_{-1}}.
    $$
\end{theorem}
Эта теорема будет доказана в разделе \ref{sect:Pairs-proof}. Однако по сути идея доказательства тривиальна. Разобьем все множество координат на пары. В итоговую совокупность наберем те векторы, у которых в каждой паре обе координаты нулевые или ненулевые одновременно. Ясно, что в таком случае все возможные попарные скалярные произведения будут четными, и нечетный запрет никогда не реализуется.

\section{Численные результаты}\label{sect:numerical-results}

На рисунках \ref{fig:full-comparison-k_1-0.49} и \ref{fig:full-comparison-k_1-0.25} мы приводим численные результаты теорем \ref{th:FW}--\ref{th:Glue} и \ref{th:Pairs} при $k'_0 = 0.5$. Первый из них касается случая $k_{-1}' = 0.01, k_1' = 0.49$, второй --- случая $k_{-1}' = k_1' = 0.25$. Все оценки из  упомянутых выше теорем могут быть записаны в экспоненциальном виде $(\lambda(k'_{-1}, k'_0, k'_1, t') + o(1))^n$ при $n\to\infty$, поэтому по оси абсцисс откладывается параметр $t'$, а по оси ординат --- константа $\lambda(k'_{-1}, k'_0, k'_1, t')$ в основании экспоненты соответствуюшей оценки. На данных рисунках отсутствует оценка теоремы \ref{th:SuperGlue}, поскольку технически сложно вычислить ее значения для всех возможных $t'$. Вместо этого мы приведем таблицу \ref{tab:full-comparison-k_1-0.495}, в которой показывается, что в узком диапазоне значений данная оценка является лучшей из представленных нижних оценок. 

\begin{figure}[!ht]
\centering
\includegraphics[width=0.75\textwidth]{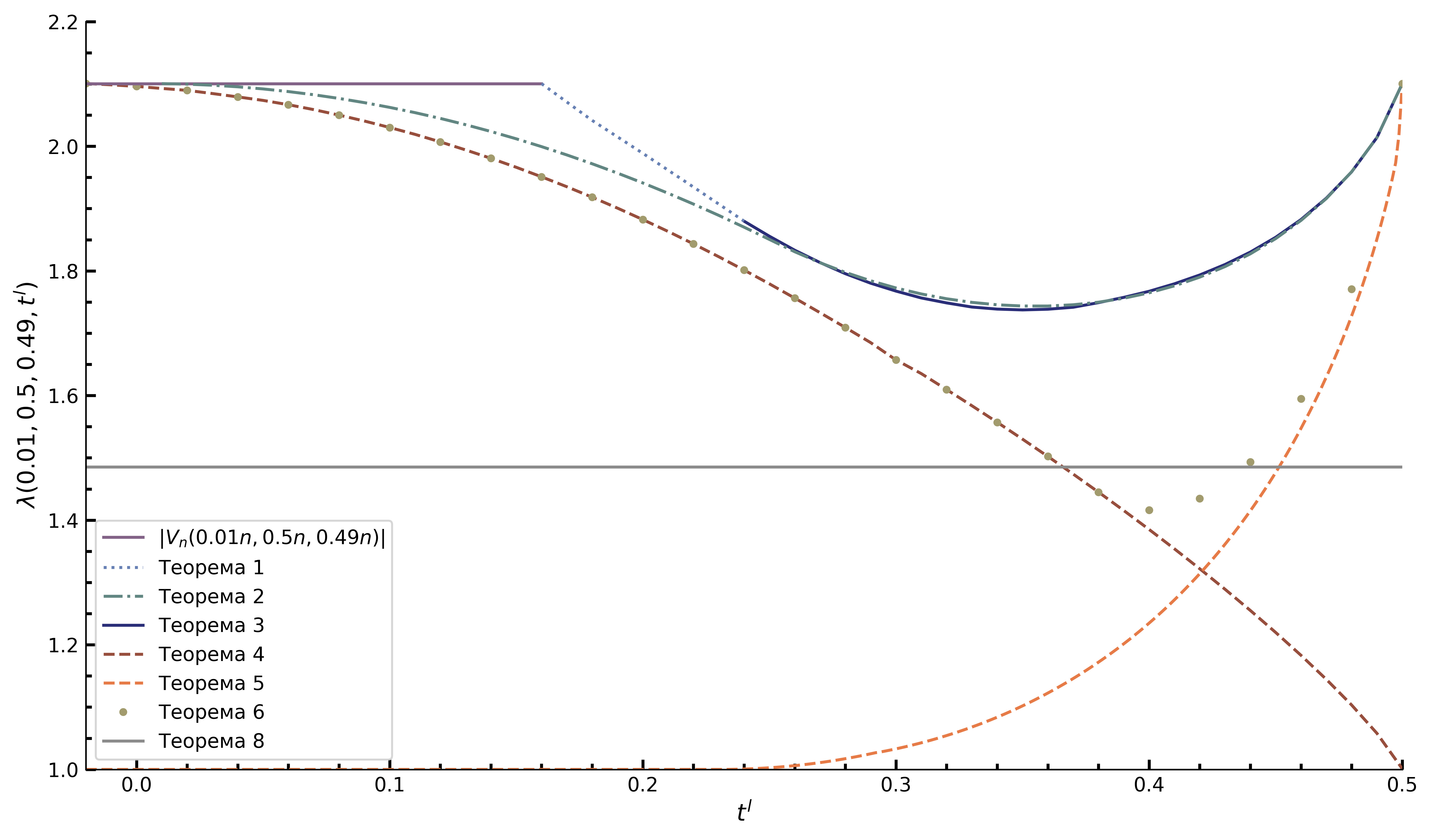}
\caption{Сравнение оценок теорем \ref{th:FW}--\ref{th:Glue} и \ref{th:Pairs} при $k_{-1}' = 0.01, k'_0 = 0.5, k_1' = 0.49$}
\label{fig:full-comparison-k_1-0.49}
\end{figure}

\begin{figure}[!ht]
\centering
\includegraphics[width=0.75\textwidth]{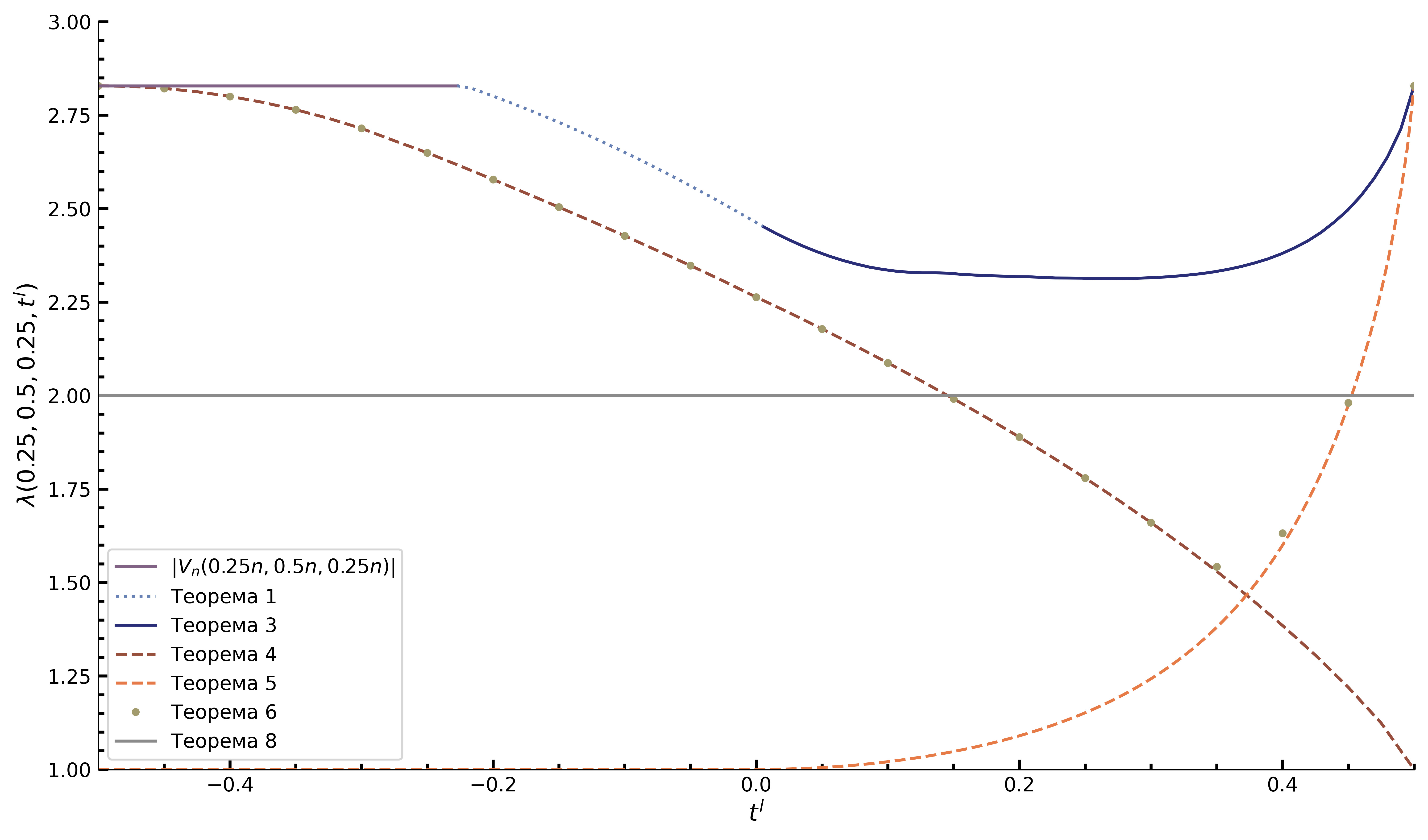}
\caption{Сравнение оценок теорем \ref{th:FW}--\ref{th:Glue} и \ref{th:Pairs} при $k_{-1}' = 0.25, k'_0 = 0.5, k_1' = 0.25$}
\label{fig:full-comparison-k_1-0.25}
\end{figure}

\begin{table}[!ht]
\begin{center}
\begin{tabular}{ | c | c | c | c | c | c | c | c | } 
  \hline
  $t'$ & $0.37$ & $0.373$ & $0.376$ & $0.379$ & $0.382$ & $0.385$ \\ 
  \hline
  Теорема \ref{th:decreasing} & $\mathbf{1.47408}$ & $1.46544$ & $1.45676$ & $1.44806$ & $1.43927$ & $1.43043$ \\ 
  \hline
  Теорема \ref{th:Glue} & $1.47408$ & $1.46544$ & $1.45698$ & $1.44930$ & $1.44231$ & $1.43509$ \\ 
  \hline
  Теорема \ref{th:SuperGlue} & $1.47408$ & $\mathbf{1.46831}$ & $\mathbf{1.46267}$ & $\mathbf{1.45644}$ & $1.45157$ & $1.44618$ \\ 
  \hline
  Теорема \ref{th:Pairs} & $1.45437$ & $1.45437$ & $1.45437$ & $1.45437$ & $\mathbf{1.45437}$ & $\mathbf{1.45437}$ \\ 
  \hline
\end{tabular}
\caption{Сравнение оценок теорем \ref{th:decreasing}, \ref{th:Glue}, \ref{th:SuperGlue} \ref{th:Pairs} при $k_{-1}' = 0.005, k'_0 = 0.5, k_1' = 0.495$}
\label{tab:full-comparison-k_1-0.495}
\end{center}
\end{table}

Обсудим сперва верхние оценки, то есть теоремы \ref{th:FW}--\ref{th:PonRai}. Во-первых, отметим, что при малых значениях $t'$ оценка теоремы \ref{th:FW} оказывается больше числа всех вершин, поэтому на графике мы заменяем ее на $|V_n(k_{-1}, k_0, k_1)|$. Во-вторых, из формулировок теорем \ref{th:FW} и \ref{th:PonRai} видно, что они работают в непересекающихся диапазонах параметра $t'$. Поэтому на графике оценка теоремы \ref{th:FW} плавно переходит в оценку теоремы \ref{th:PonRai}. Наконец, оказывается, что теорема \ref{th:FW-plus-flower} совсем не работает при $k'_1 = k'_{-1}$, именно поэтому мы не приводим ее на рисунке \ref{fig:full-comparison-k_1-0.25}. Но вот при $k_{-1}' = 0.01, k_1' = 0.49$ теорема \ref{th:FW-plus-flower} уступает теореме \ref{th:PonRai} только на отрезке $[0.28, 0.38]$.

Перейдем к нижним оценкам. Ожидаемо оказывается, что конструкция Альсведе--Хачатряна хорошо работает при небольших значениях $t'$, поскольку попарные скалярные произведения в ней заведомо больше $t'n$. Напротив, конструкция Варшамова--Гилберта оптимальна при значениях $t'$, близких к максимальному, ведь векторы из этой совокупности имеют попарные скалярные произведения меньше $t'n$. Теорема \ref{th:Glue} также довольно ожидаемо улучшает обе эти конструкции при средних значениях $t'$, где становится выгодно иметь скалярные произведения и меньше, и больше $t'n$. Довольно примечательно, что специфическая теорема \ref{th:Pairs} в определенном диапазоне значений $t'$ оказывается самой сильной. Разумеется, данная оценка совсем не зависит от параметра $t'$, поэтому на графике изображается просто горизонтальной линией. 

Теорема \ref{th:SuperGlue} дает наилучшие результаты при $k'_1 \gg k'_{-1}$ в окрестности точки, где теорема \ref{th:Pairs} начинает доминировать над теоремой \ref{th:decreasing}. Улучшения есть, но они невелики и верхних оценок, конечно, не достигают. Поэтому в таблице оценки из теорем \ref{th:FW}--\ref{th:PonRai} мы не приводим, как и оценку теоремы \ref{th:VG}, которая на интересующем нас участке очень слаба.

Любопытно, что в случае $(-1, 0, 1)$-векторов верхние оценки устроены несколько иначе, чем при отсутствии минус единиц. Например, в работе \cite{BobuKuprRai} показано, что при $t < k/2$ (здесь $k$ --- количество единиц) мощность конструкции Альсведе--Хачатряна и верхняя оценка, аналогичная оценке теоремы \ref{th:FW}, имеют одинаковую константу в основании экспоненты. В нашем случае по крайней мере при $k_1 = k'_{-1}$ и малых $t'$ верхняя оценка является ни чем иным, как числом всех вершин. Это может быть связано с тонкостями линейно-алгебраического метода: переход к $(-1, 0, 1)$-векторам автоматически увеличивает размерность пространства полиномов в доказательстве и значительно ухудшает оценку. Практически нет сомнений, что как минимум при малых $t'$ оценка теоремы \ref{th:FW} не является оптимальной и может быть улучшена. При больших же $t'$ доминирующей становится оценка теоремы \ref{th:PonRai}. К сожалению, зазор между ней и нижними оценками по-прежнему остается экспоненциальным для всех значений $t'$, кроме тривиального.

\section{Доказательства}\label{sect:proofs}

\subsection{Лемма о минимальном скалярном произведении}\label{sect:lem-mdp}

\begin{lemma}\label{lem:mdp}
    Минимальное скалярное произведение векторов из $V_n(k_1, k_0, k_1)$ вычисляется как
    $$
      \mdp(k_{-1}, k_0, k_1) =
        \begin{cases}
          -2\min(k_{-1}, k_1), \: \text{если $k_0 \geqslant \max(k_{-1}, k_1) - \min(k_{-1}, k_1)$,} \\
          \max(k_{-1}, k_1)  - 3\min(k_{-1}, k_1) - k_0, \: \text{иначе.}
        \end{cases}       
    $$
\end{lemma}
\begin{proof}
    Для простоты будем сначала считать, что $k_1 > k_{-1}$, противоположный случай рассматривается аналогично. При этом предположении условие $k_0 \geqslant \max(k_{-1}, k_1) - \min(k_{-1}, k_1)$ переписывается как $k_0 \geqslant k_1 - k_{-1}$. Разберем два возможных варианта. 

    \textbf{Случай 1:} $k_0 \geqslant k_1 - k_{-1}$. 
    
    Для начала заметим, что скалярное произведение двух векторов из $V_n(k_{-1}, k_0, k_1)$ не может быть меньше $-2\min(k_{-1}, k_1)$. Если $k_0 \geqslant k_1 - k_{-1}$, то минимальное скалярное произведение действительно реализуется, например, на следующей паре векторов:
    \begin{align*}
        \mathbf x &= \{\underbrace{-1, \ldots, -1}_{k_{-1}}, \underbrace{0, 0, \ldots, 0}_{k_0}, \underbrace{1, 1, 1, 1, \ldots, 1}_{k_1}\}; \\
        \mathbf y &= \{\underbrace{1, 1, 1, 1, \ldots, 1}_{k_1}, \underbrace{0, 0,  \ldots, 0}_{k_0}, \underbrace{-1, \ldots, -1}_{k_{-1}}\}.
    \end{align*}
    Заметим, что условие $k_0 \geqslant k_1 - k_{-1}$ существенно. Если оно не выполняется, то ``лишние'' единицы оказываются друг напротив друга, и скалярное произведение увеличивается. 

    \textbf{Случай 2:} $k_0 < k_1 - k_{-1}$. 

    Очевидно, в этом случае будет как минимум $k_1 - k_{-1} - k_0$ координат, которые равны 1 для обоих векторов. Минимально возможный вклад других координат в итоговое скалярное произведение по-прежнему равен $-2\min(k_{-1}, k_1) = -2k_{-1}$. Таким образом, итоговое скалярное произведение не может быть меньше $k_1 - k_{-1} - k_0 - 2k_{-1} = k_1 - 3k_{-1} - k_0$. Оно также достигается на векторах $\mathbf x$ и $\mathbf y$, приведенных выше. 

    При $k_1 \leqslant k_{-1}$ все приведенные выше рассуждения проводятся аналогично. 
\end{proof}

\subsection{Доказательство теоремы \ref{th:FW-plus-flower}}\label{sect:FW-plus-flower-proof}

Идея доказательства данной теоремы достаточно проста и во многом уже была реализована в работе \cite{Rai}. Она состоит в том, чтобы ``зафиксировать'' произвольные $k_{-1}$ координат в множестве векторов, в котором не реализуется интересующий нас запрет. Оставшаяся совокупность будет состоять по сути из $(0,1)$-векторов, и к ней уже можно применить стандартные верхние оценки, получаемые при помощи линейно-алгебраического метода. Тем самым мы избавляемся от автоматического ухудшения оценок линейно-алгебраического метода при применении к $(-1, 0, 1)$-векторам. 

Итак, пусть в некотором множестве $(-1, 0, 1)$-векторов $\mathcal W \subset V_n(k_{-1}, k_0, k_1)$ не реализуется скалярное произведение $t$. Для произвольного $k_{-1}$-элементного подмножества координат $M$ рассмотрим совокупность $\mathcal W_M$, состоящую из векторов $\mathbf{x} = (x_1, \ldots, x_n) \in W$ таких, что их минус единицы в точности совпадают с множеством $M$, то есть 
$$
    \mathcal W_M = \bigl \{(x_1, \ldots, x_n) \in \mathcal W \colon \{i \colon x_i = -1\}  = M \bigr\}.
$$
Рассмотрим теперь $(k_1 + k_0)$-мерные векторы, полученные из векторов совокупности $\mathcal W_M$ выкидыванием фиксированных координат из множества $M$. Очевидно, эти векторы имеют $k_1$ единиц и не имеют минус единиц. Кроме того, в этой совокупности не реализуется запрет $t - k_{-1}$, поскольку вклад координат из $M$ в скалярное произведение равен в точности $k_{-1}$. 

Поэтому в первом случае теоремы применим классическую теорему Франкла и Уилсона из работы \cite{FW}:
$$
    |\mathcal W_M| \leqslant \sum\limits_{i = 0}^{q-1} C_{k_1 + k_0}^i,
$$
где $q = k_1 - (t - k_{-1})$. В противоположном случае применим оценку теоремы 2 из работы \cite{PonRai1}:
$$
    |W_M| \leqslant \frac{ C^{d_2 + 2r}_{n_1} C^{d_1}_{n - k_{-1}}}{ C^{d_2 + r}_{k_2} C^{r}_{n_1 - k_2} C^{d_1}_{k_1} } \left(\sum\limits_{i = 0}^{q-1} C_{n_1}^i\right),
$$
где числа $d_1$ и $d_2$ определены в условиях теоремы. 

Остается лишь заметить, что 
$$
    |\mathcal W| \leqslant C_n^{k_{-1}} |W_M|.
$$
Доказательство теоремы завершено.

\subsection{Доказательство теоремы \ref{th:VG}}\label{sect:VG-proof}

Основная идея доказательства состоит в достаточно простом применении вероятностного метода. Для этого сперва вычислим вероятность того, что два произвольных вектора из $V_n(k_{-1}, k_0, k_1)$ имеют скалярное произведение хотя бы $t$. Зафиксируем произвольный вектор $\mathbf{x} \in V_n(k_{-1}, k_0, k_1)$ и покажем, что количество векторов $\mathbf y$, для которых $(\mathbf{x}, \mathbf{y}) \geqslant t$, равно в точности $d(n, k_{-1}, k_0, k_1, t)$. 

При $\alpha, \beta \in \{-1, 1\}$ обозначим $l_{\alpha, \beta}$ количество координат, которые равны $\alpha$ в векторе $\mathbf x$ и $\beta$ в векторе $y$ с $(\mathbf{x}, \mathbf{y}) \geqslant t$. Посчитаем количество способов выбрать вектор $\mathbf y$ при фиксированных $l_{\alpha, \beta}$. 
\begin{enumerate}
    \item Среди $k_1$ единиц вектора $\mathbf x$ выберем сначала $l_{1,1}$ общих единиц, а затем $l_{1, -1}$ координат, равных $-1$ в векторе $\mathbf y$. Очевидно, число таких способов равно $C^{l_{1,1}}_{k_1} C^{l_{1,-1}}_{k_1-l_{1,1}}$. 
    \item Аналогично среди $k_{-1}$ отрицательных координат вектора $\mathbf x$ выберем единицы и минус единицы вектора $\mathbf y$. Число способов это сделать \\ $C^{l_{-1,-1}}_{k_{-1}} C^{l_{-1,1}}_{k_{-1}-l_{-1,-1}}$. 
    \item Наконец, среди $k_0$ нулей вектора $\mathbf x$ необходимо сперва разместить $k_{1} - l_{1,1} - l_{-1,1}$ единиц вектора $y$. Среди оставшихся $k_{0} - k_{1} + l_{1,1} + l_{-1,1}$ координат будет $k_{-1} - l_{-1,-1} - l_{1,-1}$ минус единиц вектора $\mathbf y$. Число способов выбора равно $C^{k_{1} - l_{1,1} - l_{-1,1}}_{k_{0}} C^{k_{-1} - l_{-1,-1} - l_{1,-1}}_{k_{0} - k_{1} + l_{1,1} + l_{-1,1}}$. 
\end{enumerate}

Поскольку для позиций нулей вектора $\mathbf y$ существует лишь единственный способ, общее количество способов выбрать вектор $\mathbf y$ при фиксированных $l_{\alpha, \beta}$ равно
\begin{multline*}
    C^{l_{1,1}}_{k_1} C^{l_{1,-1}}_{k_1-l_{1,1}} C^{l_{-1,-1}}_{k_{-1}} C^{l_{-1,1}}_{k_{-1}-l_{-1,-1}} C^{k_{1} - l_{1,1} - l_{-1,1}}_{k_{0}} C^{k_{-1} - l_{-1,-1} - l_{1,-1}}_{k_{0} - k_{1} + l_{1,1} + l_{-1,1}} \times \\ 
    \times P(l_{1,1}, l_{-1,1}, l_{-1,-1}, l_{1,-1}, t).
\end{multline*}
Заметим, что параметры $l_{\alpha, \beta}$ однозначно устанавливают скалярное произведение пары векторов. Поэтому индикатор $P(l_{1,1}, l_{-1,1}, l_{-1,-1}, l_{1,-1}, t)$ показывает, что данный набор $l_{\alpha, \beta}$ действительно гарантирует $(\mathbf{x}, \mathbf{y}) \geqslant t$. Наконец, совокупность $\mathcal B$ описывает все возможные наборы параметров $l_{\alpha, \beta}$. Таким образом, общее количество способов выбрать вектор $\mathbf y$, дающий  $(\mathbf{x}, \mathbf{y}) \geqslant t$, действительно равно $d(n, k_{-1}, k_0, k_1, t)$.

Следовательно, вероятность, что два случайных вектора из $V_n(k_{-1}, k_0, k_1)$ имеют скалярное произведение не меньше $t$, равна
$$
    p = \frac{d(n, k_{-1}, k_0, k_1, t)}{|V_n(k_{-1}, k_0, k_1)|}.
$$
Пусть теперь $\mathcal{S}$ --- наибольшее множество, не содержащее пар векторов со скалярным произведением как минимум $t$. Рассмотрим  случайный вектор $\mathbf x \in V_n(k_{-1}, k_0, k_1)$. Для любого $\mathbf y \in \mathcal S$ вероятность того, что $(\mathbf x, \mathbf y) \geqslant t$, равна $p$. 

Пусть $q$ --- вероятность того, что $\mathbf x$ имеет скалярное произведение как минимум $t$ хотя бы с одним представителем $\mathcal{S}$. Данную вероятность можно оценить сверху: 
$$
    q \leqslant |\mathcal{S}|p.
$$
Очевидно, что при $|\mathcal{S}| < \lceil \frac 1p \rceil$
$$
    q \leqslant |\mathcal{S}|p < 1.
$$
Поскольку данная вероятность строго меньше 1, найдется вектор из $V_n(k_{-1}, k_0, k_1)$, скалярное произведение которого со всеми представителями $\mathcal S$ меньше $t$. Очевидно, этот вектор не может принадлежать $S$, иначе он имел бы максимальное скалярное произведение с самим собой. То есть $\mathcal{S}$ можно увеличить, что противоречит его определению. Отсюда $|\mathcal{S}| \geqslant \lceil \frac 1p \rceil$, то есть 
$$
    h(n, k_{-1}, k_0, k_1, t) \geqslant \left\lceil \frac 1p \right\rceil.
$$
Неравенство $m(n, k_{-1}, k_0, k_1, t) \geqslant h(n, k_{-1}, k_0, k_1, t)$ тривиально. 

\subsection{Доказательство теоремы \ref{th:Glue}}\label{sect:Glue-proof}

Для начала заметим, что по теореме \ref{th:VG} существует такая совокупность $\mathcal F$  векторов из $V_n(m_{-1}, m_0, m_1)$, что для любых двух $\mathbf x, \mathbf y \in \mathcal F$ выполнено $(\mathbf x, \mathbf y) < t_1$, при этом $|\mathcal F| = h(n,m_{-1}, m_0, m_1, t_1)$. 

Для фиксированного вектора $\mathbf{x} = (x_1, \ldots, x_n) \in \mathcal F$ обозначим $\mathcal  W_{\mathbf x}$ совокупность таких векторов $\mathbf{u}=(u_1, \ldots, u_n) \in V_n(k_{-1}, k_0, k_1)$, что
$$
    \left|\bigl\{i \in \{1, \ldots, n\} \colon u_i = \alpha \wedge x_i = \beta \bigr\}\right| = m_{\alpha, \beta}, \text{ где $\alpha, \beta \in \{-1, 0, 1\}$}.
$$
Отметим, что сам вектор $\mathbf x$ не обязательно принадлежит к совокупности $\mathcal W_{\mathbf x}$.  Положим $\displaystyle \mathcal W = \bigcup_{\mathbf{x} \in \mathcal{F}} \mathcal W_{\mathbf{x}}$ и убедимся, что в совокупности $\mathcal W$ скалярное произведение любых двух векторов не равняется $t$. 

Из ограничения на параметры $ \sum^1_{\beta = -1} \mdp(m_{-1, \beta}, m_{0, \beta}, m_{1, \beta}) > t$ следует, что для любого  $\mathbf{x} \in \mathscr{F}$ и любых векторов $\mathbf{u}, \mathbf{v} \in \mathcal W_{\mathbf{x}}$ выполнено $(\mathbf{u}, \mathbf{v}) > t$ (доказательство этого факта полностью почти дословно повторяет комментарий к теореме \ref{th:decreasing}, поэтому мы опустим его). 

Зафиксируем теперь различные $\mathbf{x}, \: \mathbf{y} \in \mathcal F$. Покажем, что для любых $\mathbf{u} \in \mathcal W_{\mathbf{x}} , \: \mathbf{v} \in \mathcal W_{\mathbf{y}}$ выполнено $(\mathbf{u}, \mathbf{v}) < t$. Положим
\begin{align*}
    L_{\alpha, \beta} &= \{i \colon x_i = \alpha \land y_i = \beta\}, \quad \alpha, \beta \in \{-1, 0, 1\}; \\
    l_{\alpha, \beta} &= \left|L_{\alpha, \beta}\right|, \quad \alpha, \beta \in \{-1, 0, 1\}.
\end{align*}
Оценим максимально возможное скалярное произведение векторов $\mathbf u$ и $\mathbf v$. Чтобы сделать это, оценим максимально возможный вклад в итоговое скалярное произведение в каждой из частей $L_{\alpha, \beta}$. 
\begin{enumerate}
    \item \textbf{Множества $L_{1,1}$ и $L_{-1,-1}$.} Скалярное произведение $\mathbf{u}$ и $\mathbf{v}$ среди координат $L_{1,1}$ не превосходит общего количества этих координат, то есть
    \begin{equation}\label{eq:L11}
        \sum\limits_{i \in L_{1, 1}} u_i v_i \leqslant l_{1, 1}.
    \end{equation}
    Аналогично 
    \begin{equation}\label{eq:L-1-1}
        \sum\limits_{i \in L_{-1, -1}} u_i v_i \leqslant l_{-1, -1}.
    \end{equation}
    \item \textbf{Множества $L_{1,-1}$ и $L_{-1,1}$.} Рассмотрим сначала случай множества $L_{1,-1}$ и представим три способа оценить вклад этих координат в скалярное произведение. Во-первых, заметим, что $l_{1, -1} \leqslant \min(m_{-1}, m_1)$, то есть $\min(m_{-1}, m_1) - l_{1, -1} \geqslant 0$. Поэтому 
    \begin{gather*}
        \sum\limits_{i \in L_{1, -1}} u_i v_i \leqslant l_{1, -1} \leqslant \min(m_{-1}, m_1) \leqslant  \\ \leqslant \min(m_{-1}, m_1) + \min(m_{-1}, m_1) - l_{1, -1} = -l_{1, -1} + 2\min(m_{-1}, m_1). 
    \end{gather*}
    Этот же вклад в скалярное произведение можно оценить и другим способом. Вектор $\mathbf u$ имеет самое большее $m_{-1, 1}$ минус единиц в множестве $L_{1, -1}$, а вектор $\mathbf v$ максимум $m_{-1, -1}$ минус единиц в том же множестве. Таким образом, количество координат из $L_{1, -1}$, равных $-1$ в векторах $\mathbf u$ и $\mathbf v$ одновременно, не превышает $\min(m_{-1, 1}, m_{-1, -1})$. Аналогично количество общих единиц в этих векторах не больше $\min(m_{1, 1}, m_{1, -1})$. Вклад в скалярное произведение можно оценить суммой этих двух величин: 
    \begin{equation}\label{eq:second-bound-L1-1}
        \sum\limits_{i \in L_{1, -1}} u_i v_i \leqslant \min(m_{-1, 1}, m_{-1, -1}) + \min(m_{1, 1}, m_{1, -1}).
    \end{equation}
    Используя уже известное неравенство  $\min(m_{-1}, m_1) - l_{1, -1} \geqslant 0$, имеем 
    \begin{multline}
        \sum\limits_{i \in L_{1, -1}} u_i v_i \leqslant \min(m_{-1, 1}, m_{-1, -1}) + \min(m_{1, 1}, m_{1, -1}) + \\
        + \min(m_{-1}, m_1) - l_{1, -1}.\label{eq:second-bound-L1-1-step2}
    \end{multline}
    Третий способ оценки модифицирует предыдущий. Заметим, что у вектора $\mathbf u$ не менее $l_{1,-1} - m_{1, 1} - m_{0, 1}$ минус единиц, а у вектора $\mathbf v$ не менее $l_{1,-1} - m_{-1, -1} - m_{0, -1}$ единиц в $L_{1, -1}$. Следовательно, у этой пары векторов не менее 
    \begin{multline*}
        l_{1,-1} - m_{1, 1} - m_{0, 1} + l_{1,-1} - m_{-1, -1} - m_{0, -1} - l_{1, -1} = \\ l_{1,-1} - m_{1, 1} - m_{0, 1} - m_{-1, -1} - m_{0, -1}
    \end{multline*}
    координат с разным знаком. Они, разумеется, вносят отрицательный вклад в итоговое скалярное произведение. Аналогичное рассуждение легко провести для единиц вектора $\mathbf u$ и минус единиц вектора $\mathbf v$ в $L_{1, -1}$. Улучшая неравенство \eqref{eq:second-bound-L1-1}, получаем
    \begin{multline*}
       \sum\limits_{i \in L_{1, -1}} u_i v_i \leqslant \min(m_{-1, 1}, m_{-1, -1}) + \min(m_{1, 1}, m_{1, -1}) -\\
        - \max(0, l_{1,-1} - m_{1, 1} - m_{0, 1} - m_{-1, -1} - m_{0, -1}, l_{1,-1} - m_{-1, 1} - m_{0, 1} - m_{1, -1} - m_{0, -1}). 
    \end{multline*}
    Оценивая максимум через максимум второго и третьего аргументов, имеем
    \begin{align*}
        \sum\limits_{i \in L_{1, -1}} u_i v_i &\leqslant \min(m_{-1, 1}, m_{-1, -1}) + \min(m_{1, 1}, m_{1, -1}) +\\
        + \min(m_{1, 1} &+ m_{-1, -1}, m_{-1, 1} + m_{1, -1}) + m_{0, 1} + m_{0, -1} - l_{1,-1}. 
    \end{align*}
    Выбирая меньшую из трех оценок, получаем в точности 
    $$
        -l_{1, -1} + extras(m_{0, -1}, m_{0, 1}, m_{1, -1}, m_{1, 1},  m_{-1, -1}, m_{-1, 1}, m_{-1}, m_1),
    $$ 
    поэтому окончательно
    \begin{multline}
        \sum\limits_{i \in L_{1, -1}} u_i v_i \leqslant -l_{1, -1} + \\
        + extras(m_{0, -1}, m_{0, 1}, m_{1, -1}, m_{1, 1},  m_{-1, -1}, m_{-1, 1}, m_{-1}, m_1). \label{eq:L1-1}
    \end{multline}
    Аналогично, 
    \begin{multline}
        \sum\limits_{i \in L_{-1, 1}} u_i v_i \leqslant -l_{-1, 1} + \\ 
        + extras(m_{0, -1}, m_{0, 1}, m_{1, -1}, m_{1, 1},  m_{-1, -1}, m_{-1, 1}, m_{-1}, m_1). \label{eq:L-11}
    \end{multline}
    \item \textbf{Множество $L_0 = \bigcup\limits_{\alpha = -1}^1 L_{\alpha,0} \cup \bigcup\limits_{\beta = -1}^1 L_{0, \beta}$.} В данном множестве координат по крайней мере у одного из векторов $\mathbf x$ или $\mathbf y$ координата равна нулю. В таком случае вклад этих координат в скалярное произведение можно оценить достаточно грубо: 
    \begin{equation}\label{eq:L0}
        \sum\limits_{i \in L_0} u_i v_i \leqslant 2(m_{1,0} + m_{-1, 0}). 
    \end{equation}
    На этом разбор случаев завершается. 
\end{enumerate}

Оценим теперь итоговое скалярное произведение. Собирая вместе неравенства \eqref{eq:L11}, \eqref{eq:L-1-1} и \eqref{eq:L1-1}--\eqref{eq:L0}, получаем, что 
\begin{multline*}
    (\mathbf u, \mathbf v) \leqslant l_{1, 1} + l_{-1, -1} - l_{1, -1} - l_{-1, 1} + \\ 2 \cdot extras(m_{0, -1}, m_{0, 1}, m_{1, -1}, m_{1, 1},  m_{-1, -1}, m_{-1, 1}, m_{-1}, m_1)+ 2(m_{1,0} + m_{-1, 0}). 
\end{multline*}
Однако по построению совокупности $\mathcal F$ верно, что $l_{1, 1} + l_{-1, -1} - l_{1, -1} - l_{-1, 1} < t_1$. Таким образом, 
$$
    (\mathbf u, \mathbf v) < t_1 + 2 \cdot extras(m_{0, -1}, m_{0, 1}, m_{1, -1}, m_{1, 1},  m_{-1, -1}, m_{-1, 1}, m_{-1}, m_1) + 2(m_{1,0} + m_{-1, 0}).
$$
Остается лишь заметить, что по условию теоремы правая часть не превышает $t$, поэтому 
$$
    (\mathbf u, \mathbf v) < t,
$$
как и требовалось. Кроме того, из данного неравенства автоматически следует, что совокупности $\mathcal W_{\mathbf x}$ и $\mathcal W_{\mathbf y}$ не пересекаются для различных $\mathbf x$ и $\mathbf y$. 

Остается заметить, что величина $|\mathcal W_{\mathbf x}|$ не зависит от $\mathbf x$, и 
$$
|\mathcal W_{\mathbf x}| = \prod^1_{\beta = -1} C^{m_{-1, \beta}}_{m_{\beta}} C^{m_{0, \beta}}_{m_{0, \beta} + m_{1, \beta}} \text{ для любого $\mathbf x \in \mathcal F$}. 
$$
Мы опустим эти вычисления, поскольку они, во-первых, совершенно тривиальны, а во-вторых, могут быть найдены в доказательстве теоремы \ref{th:decreasing} в работе \cite{MosRai}. Таким образом, 
$$
    |\mathcal W| = |\mathcal F| \cdot |\mathcal W_{\mathbf x}| = h(n,m_{-1}, m_0, m_1, t_1) \prod^1_{\beta = -1} C^{m_{-1, \beta}}_{m_{\beta}} C^{m_{0, \beta}}_{m_{0, \beta} + m_{1, \beta}}.
$$

\subsection{Доказательство теоремы \ref{th:SuperGlue}}\label{sect:proof-superglue}

Как и в случае прошлого доказательства, теорема \ref{th:VG} гарантирует наличие такой совокупности $\mathcal F$ векторов из \\ $V_n(m_{-1}, m_0, m_1)$, что для любых двух $\mathbf x, \mathbf y \in \mathcal F$ выполнено $(\mathbf x, \mathbf y) < s$, при этом $|\mathcal F| = h(n,m_{-1}, m_0, m_1, s)$.

Аналогично предыдущему доказательству, для фиксированного вектора $\mathbf{x} = (x_1, \ldots, x_n) \in \mathcal F$ обозначим $\mathcal  W_{\mathbf x}$ совокупность таких векторов $\mathbf{u}=(u_1, \ldots, u_n) \in V_n(k_{-1}, k_0, k_1)$, что
$$
    \left|\bigl\{i \in \{1, \ldots, n\} \colon u_i = \alpha \wedge x_i = \beta \bigr\}\right| = m_{\alpha, \beta}, \text{ где $\alpha, \beta \in \{-1, 0, 1\}$}.
$$

Однако итоговую совокупность, в которой не будет встречаться запрета $t$, будем формировать несколько иначе. А именно, обозначим
\begin{multline*}
    \mathcal S_{\mathbf x} = \bigcap\limits_{\mathbf y \in \mathcal F} \{ \mathbf u \in \mathcal W_{\mathbf x} \colon \mathbf u \text{ имеет менее $t - 2(m_{1,0} + m_{-1,0})$ ненулевых координат} \\ 
    \text{ среди координат, отличных от нуля в $\mathbf x$ и $\mathbf y$ одновременно} \}. 
\end{multline*}
Итоговая совокупность будет определяться как 
$$
    \mathcal S = \bigcup\limits_{\mathbf x \in \mathcal F} \mathcal S_{\mathbf x}.
$$

Как и в случае доказательства теоремы \ref{th:Glue}, векторы из одной подсовокупности $\mathcal S_{\mathbf x}$ имеют попарные скалярные произведения строго больше $t$. Несложно также показать, что если $\mathbf u \in \mathcal S_{\mathbf x}$, а $\mathbf v \in \mathcal S_{\mathbf y}$, то $(\mathbf u, \mathbf v) < t$. Разделим множество координат на два множества. Обозначим $L_0$ множество координат, равных нулю по крайней мере в одном из векторов $\mathbf x$ или $\mathbf y$, и $L^*_0$ --- множество всех остальных координат. Аналогично доказательству предыдущей теоремы, 
$$
    \sum\limits_{i \in L_0} u_i v_i \leqslant 2(m_{1,0} + m_{-1, 0}). 
$$
С другой стороны, построение совокупности $\mathcal S_{\mathbf x}$ гарантирует, что 
$$
    \sum\limits_{i \in L^*_0} u_i v_i < t - 2(m_{1,0} + m_{-1, 0}). 
$$
Отсюда 
$$
    (\mathbf u, \mathbf v) \leqslant \sum\limits_{i \in L_0} u_i v_i + \sum\limits_{i \in L^*_0} u_i v_i < t,
$$
как и требовалось. Таким образом, никакие два вектора из совокупности $\mathcal S$ не дают скалярное произведение $t$. Кроме того, отсюда следует, что совокупности $\mathcal S_{\mathbf x}$ и $\mathcal S_{\mathbf y}$ не пересекаются при различных $\mathbf x$ и $\mathbf y$.

Вычислим размер совокупности $\mathcal S$. Для этого покажем, что 
\begin{multline*}
    |\mathcal W_{\mathbf x} \backslash \mathcal S_{\mathbf x}| \leqslant h(n, m_{-1}, m_0, m_1, s) C_{m_0}^{m_{-1, 0}} C_{m_{0,0} + m_{1,0}}^{m_{0,0}} \times \\ 
    \times \max\limits_{m_{-1} + m_1 - m_0 \leqslant l \leqslant \min(s + 4\min(m_{-1}, m_1), m_{-1} + m_1)} \sum\limits_{j = t - 2(m_{1,0} + m_{-1, 0})}^l \sum\limits_{i=0}^j C_l^j C_j^i \times \\
\times C_{m_{-1} + m_1 - l}^{m_{1, 1} + m_{1, -1} + m_{-1, -1} + m_{-1, 1} - j} C_{m_{1, 1} + m_{1, -1} + m_{-1, -1} + m_{-1, 1} - j}^{m_{1, 1} + m_{1, -1} - i}.
\end{multline*} 

Сперва зафиксируем вектор $\mathbf y \in \mathcal F$ и введем следующие обозначения:  
$$
    l_{\alpha, \beta} = |\{i \in \{1, \ldots, n\} \colon x_i = \alpha \wedge y_i = \beta\}|, \quad \alpha, \beta \in \{-1, 0, 1\}.
$$
Также мы будем использовать обозначения $L_0$ и $L^*_0$, данные выше, и положим $l = |L^*_0|$. Тогда можно сформулировать, что некий вектор $\mathbf u$ принадлежит к совокупности ``плохих'' векторов $\mathcal W_{\mathbf x} \backslash \mathcal S_{\mathbf x}$, если имеет как минимум $t - 2(m_{1,0} + m_{-1, 0})$ ненулевых координат среди координат из множества $L^*_0$. Обозначим количество таких координат $j$, тогда, очевидно, $t - 2(m_{1,0} + m_{-1, 0}) \leqslant j \leqslant l$. Среди $j$ ненулевых координат будем также выбирать $i$ единиц. Тогда общее количество способов выбрать ненулевые координаты вектора $\mathbf u$ в множестве $L^*_0$ равно
$$
    \sum\limits_{j = t - 2(m_{1,0} + m_{-1, 0})}^l \sum\limits_{i=0}^j C_l^j C_j^i.
$$

В множестве $L_0$ рассмотрим сперва ненулевые координаты вектора $\mathbf x$. Расположить единицы и минус единицы вектора $\mathbf u$ при известных $j$ и $i$ среди этих координат можно 
$$
    C_{m_{-1} + m_1 - l}^{m_{1,1} + m_{1, -1} + m_{-1, 1} + m_{-1, -1} - j} C_{m_{1,1} + m_{1, -1} + m_{-1, 1} + m_{-1, -1} - j}^{m_{1,1} + m_{1,-1} - i}
$$
способами. Здесь мы воспользовались определениями величин $m_{\alpha, \beta}$ и тем фактом, что мы уже выбрали $j$ ненулевых координат и, в частности, $i$ единиц среди множества $L_0$. Наконец, значения на нулевых координатах вектора $\mathbf x$ выбираются 
$$
    C_{m_0}^{m_{-1, 0}} C_{m_{0, 0} + m_{1, 0}}^{m_{0, 0}}
$$
способами. Оставшиеся координаты восстанавливаются автоматически. Таким образом, при фиксированном векторе $\mathbf y$ количество ``плохих'' векторов из $\mathcal W_{\mathbf x} \backslash \mathcal S_{\mathbf x}$ равно 
\begin{multline*}
    C_{m_0}^{m_{-1, 0}} C_{m_{0, 0} + m_{1, 0}}^{m_{0, 0}} \sum\limits_{j = t - 2(m_{1,0} + m_{-1, 0})}^l \sum\limits_{i=0}^j C_l^j C_j^i C_{m_{-1} + m_1 - l}^{m_{1,1} + m_{1, -1} + m_{-1, 1} + m_{-1, -1} - j} \times  \\ 
    \times C_{m_{1,1} + m_{1, -1} + m_{-1, 1} + m_{-1, -1} - j}^{m_{1,1} + m_{1,-1} - i}. 
\end{multline*}

Для того, чтобы просуммировать полученное выражение по всем возможным векторам $\mathbf y \in \mathcal F$, установим неравенства на параметр $l$. Для начала заметим, что $l = l_{1,1} + l_{1, -1} + l_{-1, 1} + l_{-1, -1}$. Из построения совокупности $\mathcal F$ и определения величин $l_{\alpha, \beta}$ получаем соотношение $l + l_{1, 0} + l_{-1, 0} = m_{-1} + m_1$. С другой стороны, $l_{1, 0} + l_{-1, 0} \leqslant m_0$. Отсюда 
$$
    l = m_{-1} + m_1 - (l_{1, 0} + l_{-1, 0}) \geqslant m_{-1} + m_1 - m_0. 
$$
Что касается неравенства в противоположную сторону, то, разумеется, 
$$
    l \leqslant m_{-1} + m_1.
$$
Наконец, по построению совокупности $\mathcal F$ имеем, что $l - 2l_{1, -1} - 2l_{-1, 1} < s$. Но величины $l_{1,-1}$ и $l_{-1,1}$ не превосходят $\min(m_{-1}, m_1)$, поэтому
$$
    l < s + 2(l_{1, -1} + l_{-1, 1}) \leqslant s + 4\min(m_{-1}, m_1). 
$$
Окончательно
$$
    m_{-1} + m_1 - m_0 \leqslant l \leqslant \min(s + 4\min(m_{-1}, m_1), m_{-1} + m_1). 
$$

Суммируя теперь по всем возможным векторам $\mathbf y \in \mathcal F$, получаем, что 
\begin{multline*}
    |\mathcal W_{\mathbf x} \backslash \mathcal S_{\mathbf x}| \leqslant h(n, m_{-1}, m_0, m_1, s) C_{m_0}^{m_{-1, 0}} C_{m_{0, 0} + m_{1, 0}}^{m_{0, 0}} \times \\
    \times \max_{m_{-1} + m_1 - m_0 \leqslant l \leqslant \min(s + 4\min(m_{-1}, m_1), m_{-1} + m_1)} \sum\limits_{j = t - 2(m_{1,0} + m_{-1, 0})}^l \sum\limits_{i=0}^j C_l^j C_j^i \times \\
    \times C_{m_{-1} + m_1 - l}^{m_{1,1} + m_{1, -1} + m_{-1, 1} + m_{-1, -1} - j} C_{m_{1,1} + m_{1, -1} + m_{-1, 1} + m_{-1, -1} - j}^{m_{1,1} + m_{1,-1} - i}.
\end{multline*}
Остается отметить, что
\begin{align*}
    |S| &= \sum\limits_{\mathbf x \in \mathcal F} |\mathcal S_{\mathbf x}| = \sum\limits_{\mathbf x \in \mathcal F} (|\mathcal W_{\mathbf x}| - |\mathcal W_{\mathbf x} \backslash \mathcal S_{\mathbf x}|) \geqslant \\
    & \geqslant h(n, m_{-1}, m_0, m_1, s) \Bigl( \prod^1_{\beta = -1} C^{m_{-1, \beta}}_{m_{\beta}} C^{m_{0, \beta}}_{m_{0, \beta} + m_{1, \beta}} - \\ & h(n, m_{-1}, m_0, m_1, s) C_{m_0}^{m_{-1, 0}} C_{m_{0,0} + m_{1,0}}^{m_{0,0}} R \Bigr),
\end{align*}
где величина $R$ определена в условиях теоремы. 

\subsection{Доказательство теоремы \ref{th:Pairs}}\label{sect:Pairs-proof}

Рассмотрим пары $\{1, 2\}, \{3, 4\}, \ldots, \{n-1, n\}$. Выберем из них $(k_1 + k_{-1})/2$ пар и рассмотрим совокупность векторов из $\{-1, 0, 1\}^n$ таких, что все их ненулевые координаты содержатся в этих парах. Легко понять, что скалярное произведение любых двух векторов из этой совокупности четно, поскольку вклад в него каждой пары четен. Поскольку $t$ --- нечетное число, то запрет $t$ в построенной совокупности не реализуется. Размер полученной совокупности равен количеству способов выбрать пары и распределить в них единицы и минус единицы, то есть $C^{(k_1 + k_{-1})/2}_{n/2} C^{k_1}_{k_1 + k_{-1}}$, как и требовалось.

\end{fulltext}

\end{document}